\documentclass[[a4paper,english,12pt]{amsart}
\usepackage[utf8]{inputenc}
\usepackage{amsfonts}
\usepackage{amsmath,amsthm,amssymb}
\usepackage{bbm}
\usepackage{xcolor}

\usepackage{enumerate}

\newcommand{\HH}{{\mathcal H}}
\newcommand{\DX}{{\mathcal D}^{\infty,1}}

\newcommand{\wt}[1]{\widetilde{#1}}

\pdfpagewidth\paperwidth
\pdfpageheight\paperheight
\theoremstyle{plain}

\newtheorem{theorem}{Theorem}[section]

\newtheorem{lemma}[theorem]{Lemma}

\newtheorem{definition}[theorem]{Definition}

\newtheorem{proposition}[theorem]{Proposition}

\theoremstyle{remark}

\newtheorem{remark}[theorem]{Remark}

\newcommand\supp{\operatorname{supp}}

\usepackage{color}
\usepackage[normalem]{ulem}

\definecolor{blue(ryb)}{rgb}{0.01, 0.28, 1.0}
\definecolor{blue-green}{rgb}{0.00, 0.67, 0.9}
\definecolor{darkgreen}{rgb}{0,0.6,0.2}

\newcommand{\ft}[1]{\widehat{#1}}

\newcommand{\intt}{\mbox{\rm int\,}}

\newcommand{\Om}{{\Omega}}
\newcommand{\DD}{{\mathcal D}}

\newcommand{\FF}{\mathcal{F}}

\newcommand{\Dc}{\mathcal{D}_c}

\newcommand{\MM}{\mathcal{M}}
\newcommand{\OO}{\mathcal{O}}

\newcommand{\ve}{\varepsilon}

\newcommand{\ZZ}{\mathbb Z}

\newcommand{\KK}{\mathbb K}

\newcommand{\RR}{\mathbb R}
\newcommand{\NN}{\mathbb N}
\newcommand{\CC}{\mathbb C}

\newcommand{\oA}{\overline{A}}

\usepackage{babel}
\def\inv{\mathrm{inv}}
\def\cone{\mathrm{cone}}

\def\veps{\varepsilon}
 \def\dd{\mathrm{d}}
 \def\haar{\lambda}
 \def\car{\mathbf{1}}

\newcommand{\si}{\sigma}
\newcommand{\rh}{\rho}

\def\one{\mathbf{1}}

\begin{document}

\title[Duality for Delsarte's extremal problem on LCA groups]{Duality for Delsarte's extremal problem on locally compact Abelian groups}

\author[Berdysheva et al.]{Elena E. Berdysheva, B\'alint Farkas, Marcell Ga\'al, Mita D. Ramabulana, Szil\'ard Gy. R\'ev\'esz}


\address{Elena E. Berdysheva
\newline  \indent University of Cape Town,
\newline  \indent South Africa}

\email{elena.berdysheva@uct.ac.za}

\address{B\'alint Farkas
\newline  \indent University of Wuppertal
\newline  \indent Gau{\ss}stra{\ss}e 20, 42119 Wuppertal, Germany}

\email{farkas@math.uni-wuppertal.de}

\address{Mita D. Ramabulana
\newline  \indent University of Cape Town,
\newline  \indent South Africa}

\email{mita.ramabulana@uct.ac.za}

\address{Szil\'ard Gy. R\'evesz
\newline  \indent HUN-REN R\'enyi Institute of Mathematics,
\newline \indent Budapest, Re\'altanoda utca 13-15,1053 HUNGARY}

\email{revesz.szilard@renyi.hu}


\begin{abstract}
The Delsarte extremal problem for positive definite functions, originally introduced by Delsarte in coding theory to bound the size of error-correcting codes, has since found applications in diverse areas such as sphere packing, Fuglede's spectral set conjecture, and $1$-avoiding sets.

Recent developments have established the existence of extremizers in fairly general settings and identified precise linear programming dual formulations, together with strong duality results, in several important cases including finite groups and $\mathbb{R}^d$.

In this paper, we consider a generalized Delsarte problem on locally compact Abelian  groups, providing a natural framework for harmonic analysis. We pursue a purely functional analytic approach, which allows us to extend both the normalization and the objective functional to encompass a wide range of previously studied cases, while avoiding restrictive topological assumptions common in the literature.

Within this general setting, we derive the corresponding dual problem and prove a strong duality theorem, thereby unifying and extending earlier results. Naturally, our proof uses harmonic analysis, but the key is a functional analytic ingredient which distinguishes our proof from existing methods.
\end{abstract}
\maketitle

\section{Introduction}

The Delsarte extremal problem originates from coding theory, where the maximal possible size of codes of given length and self-correcting at a prescribed extent were estimated via their Krawtchouk polynomial expansions. See the works of Delsarte \cite{Del, Del3}, or the  survey \cite{Boy} by Boyvalenkov, Dodunekov, Musin. It turned out quickly that the approach can be reconfigured to more common setups, where non-negativity of Krawtchuk polynomial expansion coefficients corresponded to non-negativity of the Fourier transform---that is, positive definiteness. For early development of the idea we refer to Levenshtein \cite{Levenshtein}, Kabatyanskii, Levenshtein \cite{KabLev}, Ivanov \cite{Ivanov:all_h}, Yudin \cite{Yudin},  Gorbachev \cite{Gorbachev, gorbachev:ball}, Cohn, Elkies \cite{cohn:packings}.

The Delsarte extremal problem on positive definite functions for various sets in $\RR^d$ received much attention because in many instances it provides the best available tool---sometimes even a sharp tool---to access difficult problems. This was seen in case of  sphere packings, see Levenshtein \cite{Levenshtein}, Yudin \cite{Yudin}, Viazovska \cite{Viazovska}, Cohn \cite{Cohn}, Cohn, Kumar, Miller, Radchenko, Viazovska  \cite{CKMRV}, Gorbachev \cite{Gorbachev}, and in the case of the Fuglede Conjecture, too, see Lev, Matolcsi \cite{Mat-Lev}. Let us mention here also the Erd\H os-Moser distance avoiding set conjecture, see  Croft \cite{Croft1967} and Erd\H os \cite{Erdos1985}, which has seen a long development until it was proved recently in \cite{ambrus_etal} by Ambrus, Csisz\'arik, Matolcsi, Varga, and Zs\'amboki.

The Delsarte scheme has already been extended to locally compact Abelian (LCA) groups\footnote{Including the Hausdorff property.}. For example, estimates of the upper density of the translational vectors $T$ in a packing $T+\Omega$ by translates of a given set $\Omega$ are given by the Delsarte bound in LCA groups, see Berdysheva, R\'ev\'esz \cite{Berdysheva}.

The Delsarte extremal problem can be formulated as an extremal problem on a function class $\FF$. However, its formulation seems to depend on the choice of the function class, and it is far from obvious which class to consider. To be more precise, in the classical problem of sphere packing in $\RR^d$, Gorbachev in \cite{Gorbachev} worked with band-limited functions, Viazovska in \cite{Viazovska}, and again Cohn, Kumar, Miller, Radchenko, Viazovska in \cite{CKMRV} used classes of functions which, along with their Fourier transform, decay sufficiently fast (with a polynomial speed), whereas Cohn and Elkies \cite{cohn:packings} and Cohn, de Laat, Salmon \cite {Cohn-Laat-Salmon} considered the Schwartz class, etc.. The non-trivial equivalence of these, \emph{a priori} different, extremal problems is worked out by Berdysheva and R\'ev\'esz in \cite{Berdysheva} relying on useful personal communications from Gorbachev. The general definition can be formulated as follows.

\begin{definition}\label{def:Delsarte} Let $G$ be an LCA group with neutral element $0$ and Haar measure $\haar$. Denote the family of continuous positive definite functions by $\DD=\DD(G)$. Let $\FF$ be a function class, forming a subspace of $C(G;\RR)$ (real-valued continuous functions on $G$) and let $\Omega \subseteq  G$ be a symmetric set with $0\in\intt \Om$, having compact closure.

Then the \textbf{Delsarte constant} is\footnote{We remark that under the condition of $\Omega$ having compact closure, functions that are admissible for the supremum in question are automatically  integrable.}.
\begin{equation}\label{Delsarteconst}
D_G(\FF, \Omega):= \sup \left\{ \int_G f\dd \haar : f\in \DD \cap \FF, f(0)=1, f|_{G\setminus \Omega} \le 0\right\}.
\end{equation}
\end{definition}
In fact, in this work we shall generalize this setup to the respective extremal problems with both the Haar integral  (i.e., the goal functional) and the point evaluation at $0$ (i.e., normalizing functional) replaced by other almost arbitrary linear functionals.  Some  normalization for the considered functions is necessary: fully dropping any normalization constraints leaves us with a full convex cone $\FF \cap \DD$, where the goal functional $\rh $ can become unbounded.
 However, note that regarding the underlying set $\Omega$ the formulation is quite general. Indeed, we require no topological conditions on $\Omega$ apart from $0 \in \intt \Omega$, which is necessary, as otherwise the normalization condition on $f(0)=1$ trivialize the function set to $\emptyset$. Symmetry is not a restriction either, as we can always consider $\Omega \cap (-\Omega)$ instead of the original (possibly not symmetric) set, given that any positive definite function $f \in \DD$, non-positive outside $\Omega$, non-positive outside $-\Omega$, too. Finally, compact closure can be relaxed, too, as we will discuss later, but some type of boundedness is necessary, if we want $D_G(\FF,\Omega)$ to remain finite.

Obviously, the extremal problem of Delsarte is an extremal problem of infinite dimensional linear programming type, where the linear conditions bounding the feasibility set can be expressed in several ways. One direct, but rarely used variant is to say that positive definiteness is encoded in the usual inequalities
\begin{equation}\label{eqposdef2}
\sum_{k=1}^n\sum_{\ell=1}^n c_k \overline{c_\ell} f(x_k-x_\ell) \ge 0
\end{equation}
holding for all $n \in \NN,\:(c_k)_{k=1}^n \in \CC^n$, where it is important that the coefficients $c_k$ are complex. 

In fact, here is a hidden difficulty in front of us. We are doing real linear programming duality, therefore we restrict to real-valued positive definite functions, and this class is slightly different from standard positive definite functions. This issue was properly addressed in Ga\'al, R\'ev\'esz \cite{Zeitschrift}, and we will follow that work.

A second possibility is to use the non-negativity of the Fourier transform $\ft{f}$ of $f$, which is equivalent to positive definiteness whenever $\FF \subseteq  L^1(G;\RR)$, too\footnote{We use the following convention throughout: For spaces of functions over $G$ we explicitly denote the codomain, which is here either $\RR$ or $\CC$. Such precision  is important, when we consider real-valued positive definite functions.}. Nevertheless, we need to see here that our class restriction gives the same extremal values as other usual class definitions. Now, general linear programming setup requires an image space $Z$, and then we can say that the conditioning in $D_G(\FF,\Omega)$ consists of two parts, one in $X$ (the inequalities $f(x)\le 0$ on points outside $\Omega$), and another one (of non-negativity of the Fourier transform) after considering the operator of the Fourier transform $T : X \to Z$. For duality we then need to consider dual spaces $Y=X'$ and $W=Z'$, and determine the (bounded, linear) adjoint operator $T^* : W \to Y$. Then to have the strong duality result, restrictive properties of the adjoint may be needed. In fact, this ``standard linear programming approach'' was worked out first by Arestov, Babenko in \cite{AB}, and in our previous paper \cite{UFO} we followed this direction. However, we could carry over the arguments only under some compactness conditions, even if in that case in a great generality (with, e.g.,  two-sided sign conditions). Another way to look at the duality problem is to say that we do not want to leave the space $X$, so we take the operator $T$ to be the identity and change only the ``positivity cone '' to $\DD \cap X$, the cone of positive definite functions in \(X\). We are not required to deal with the individual inequalities defining the feasibility set---we need to know only the cones of positivity $P$ and $Q$ in $X$ and in its identical copy. Then the adjoint operator $T^*$ is also the identity, and we may deal with $Y=X'$ only. Even this approach is non-trivial, if we leave the realm of discrete or compact groups.


In investigations and applications of the Delsarte extremal problem,
interest was aroused regarding the linear programming dual problem of it. To establish some dual problems, which then give an upper estimate on the Delsarte constant by standard weak duality, is relatively easy (although one has to take care of the function classes here, too); see, e.g., Cohn, Elkies \cite{cohn:packings}. However, dual problems with the proof of strong duality, that is, equality of the value of the---in the linear programming terminology, ``primal''---Delsarte-type problem and the value of its dual problem, is much harder. Such results were obtained for finite Abelian groups by Matolcsi, Ruzsa in \cite{MatolcsiRuzsa}, and even earlier for the discrete group settings of $\ZZ$ by Ruzsa in \cite{Ruzsa1981} and of $\ZZ^d$ by R\'ev\'esz in \cite{R-Austral}. These results were then applied in number theory, in particular in the study of van der Corput sequences and sequences of intersectivity, by Ruzsa in \cite{Ruzsa1981}, and the distribution of Beurling primes ~\cite{Revesz-Beurling} and in an extremal problem of Landau \cite{R-Landau} by  R\'ev\'esz.

The duality result of R\'ev\'esz in \cite{R-Austral} for $\ZZ^d$ was reproduced by Virosztek \cite{virosztek} via an elegant functional analysis approach, which will be the starting point of our present approach as well. Another method to prove strong duality was developed by Arestov and Babenko in \cite{AB}. This was used to estimate the sphere kissing numbers, and is particularly relying on compactness of the underlying space. This approach was followed in our companion paper \cite{UFO} to address very general setups even in homogeneous spaces and Gelfand pairs, but also relying on the compactness assumption.

However, strong duality for $\RR^d$ resisted attempts until recently. In \cite{Cohn-Laat-Salmon} Cohn, de Laat and Salmon returned to the two decades old settings of Cohn and Elkies from \cite{cohn:packings} and proved strong duality using the Schwartz space and distributions in particular. The result was reconfirmed under a topological condition of ``continuous boundary'' by Kolountzakis, Lev, Matolcsi in \cite{KLM25}, following classical treatment of linear programming duals. However, both works relied heavily on the structure of $\RR^d$. Indeed, Cohn, de Laat, Salmon \cite{Cohn-Laat-Salmon} made heavy use of Fourier analysis in the Euclidean settings, and estimates which do not exist in general LCA groups. On the other hand, Kolountzakis, Lev, Matolcsi  \cite{KLM25} argued with the exploitation of their (otherwise, rather general) geometrical condition and along the lines of classical proofs of strong duality theorems in linear programming, referring to balls, orthogonal transformations, and the like, also missing in general groups.

Here, we generalize earlier duality results, first of all in that we extend them to general LCA groups. To make the technical part more tractable, we shall assume that the group $G$ is compactly generated, but we remark that this assumption is not needed for the validity of main result, Theorem \ref{thm:mainmain}, of this paper, see Remark \ref{rem:mainmain}. We also note that compact generation would not be a restriction either, as far as determining the Delsarte constant is concerned. In particular, it was shown by Ramabulana in \cite{ramabulanaexistence} and by Berdysheva, Ramabulana, and R\'ev\'esz in \cite{ExpoMath} that when the set of positivity has finite Haar measure, then one can restrict the Delsarte problem to a $\sigma$-compact subgroup.

We not only consider the case of general LCA groups, but  we get rid of all the topological constraints on $\Omega$ needed so far in the above mentioned various results. Furthermore, we discuss the case where the goal functional is not the Haar integral, but virtually any other linear functional, and also the normalizing constraint $f(0)=1$ is replaced by an essentially arbitrary functional $\si$. Here we need to mention that the classical normalizing condition implicitly refers to the Dirac measure at $0$, and a natural restriction to any $\si$ is that it should furnish zero value to no admitted functions---except if identically vanishing. Now we will consider some class of positive definite functions, so it is a natural requirement that $\si$ be strictly positive on the class of continuous real-valued positive definite functions $\DD$. In turn, Plancherel's theorem and Bochner's theorem together reformulate this to the famous \textbf{Wiener condition} requiring   $\ft{\si}>0$, i.e., that the Fourier transform has no zeroes on the dual group $\ft{G}$. Henceforth, we will term such a $\sigma$ \emph{strictly positive definite}.

As our main result needs even more technical introduction, let us give here the version valid for discrete Abelian groups $G$. Consider the Banach space $X=\ell^1(G;\RR)$ of real-valued  summable  functions over $G$ and take continuous linear functionals\footnote{We identify the Banach space $\ell^\infty(G;\RR)$ of real-valued  bounded  functions over $G$ with the topological dual space of $\ell^1(G;\RR)$.} $\rh,\si\in X'=\ell^\infty(G;\RR)$.
For $\Omega, \Theta \subseteq  G$ arbitrary, let us introduce the function classes
\begin{align*}
\FF_G^\si(\Omega)&:= \{ f \in \ell^1(G;\RR) ~:~ f \text{ positive definite},~ f|_{G\setminus \Omega} \le 0, \langle f,\si \rangle =1 \},
\\
\MM_G(\Theta)&:= \{ \mu \in \ell^\infty(G;\RR) : \mu=\nu-\kappa, \, \kappa\geq 0, \, \kappa|_{G\setminus\Theta}=0, \, \nu \text{ positive definite} \} .
\end{align*}

\begin{theorem}\label{thm:first}
Let $G$ be a discrete Abelian group with neutral element $0$, and let $\Omega \subseteq G$ be a symmetric set with $0\in \Omega$.  Take continuous linear functionals $\rh,\si\in \ell^\infty(G;\RR)$, where  $\si$ is assumed to be strictly positive definite
and $\rh$ is assumed to be even\footnote{For the ``primal'' problem described here the assumption that $\rh$ is even, means no loss of generality. Indeed, since a positive definite real-valued function $f\in \ell^1(G; \RR)$  is even, if  $o\in \ell^\infty(G;\RR)$ is odd, then $\langle f,o\rangle =\sum_{g\in G}f(g)o(g)=0$. So we can always consider the even part of a given $\rh\in \ell^\infty(G;\RR)$ without changing the primal problem.}.

Then to the ``primal'' linear programming extremal problem
\[
\alpha_\si^\rh (\Omega):= \inf \{\langle f,\rh\rangle~:~f \in \FF_G^\si(\Omega)\},
\]
the corresponding dual linear programming problem is
\[
\omega_\si^\rh(G\setminus\Omega):= \sup \{ s \in \RR ~:~\rh- s \si  \in \MM_G(G\setminus\Omega) \}.
\]
Moreover, there is no duality gap between the two problems.
\end{theorem}

The main aim of the paper is to generalize this to not necessarily compact or discrete, but only locally compact Abelian groups. The space, nowadays called \textbf{Wiener algebra}, in which we shall formulate the problem is of \emph{amalgam} type and was first considered by Wiener in \cite{Wiener2} for the special case of $G=\mathbb{R}$. For the extension to general locally compact Abelian groups we refer, e.g., to the works Phuong-C\'ac in \cite{Cac}, Argabright and Gil de Lamadrid \cite{AG-apm}, Holland \cite{Holland1975, Holland75b}, Stewart \cite{Stewart1979}, Feichtinger \cite{Feichtinger1977, Feichtinger1983}, Fournier, Stewart \cite{FournierStewart1985},  Bertrandias, Dupuis \cite{BD}.
More relevant bibliographical references will be given
in Section~\ref{sec:lcaback} below, where we collect the necessary technicalities. The formulation of our main result will be postponed to Section \ref{sec:main}, Theorem \ref{thm:mainmain}.

As said, in \cite{UFO} we could not extend our argument to non-compact groups, and in earlier works by Matolcsi, Ruzsa \cite{MatolcsiRuzsa}, and by R\'ev\'esz \cite{R-Austral} discreteness was heavily exploited. However, we should admit here that, on the other hand, both the methods of the paper~\cite{R-Austral} by R\'ev\'esz and those of \cite{UFO}---adapting the techniques of Arestov, Babenko \cite{AB}---were capable of handling two-sided sign restriction conditions, a notable example being the so-called \textbf{Turán problem}, a predecessor of Delsarte's problem actually introduced by Siegel in \cite{Siegel}. Our current method does not seem to provide an access to extremal problems having two-sided sign restrictions in general LCA groups, but otherwise overcomes all topological restrictions, and is absolutely general regarding the goal functional and the normalizing functional as well. What can be said in the case of discrete groups and two-sided sign conditions is described in Section \ref{sec:discrete}.

To preview the approach of the paper, we recall that everything depends on the choice of the function class and the abstract harmonic analysis related to it. The classes considered by Cohn, de Laat, Salmon in \cite{Cohn-Laat-Salmon}, by Kolountzakis, Lev, Matolcsi in \cite{KLM25} or, e.g., by Viazovska in \cite{Viazovska} are simply not available here. So, our approach is entirely different from  that in these works. We follow the paper \cite{Zeitschrift} by Ga\'al and R\'ev\'esz,  where, based on a lesser known version of the dual cone intersection formula, due to Jeyakumar and Wolkowicz \cite{JeyaWolf}, we devised a functional-analytic method to deal with inequalities for positive definite functions. The idea in itself goes back to the paper \cite{virosztek} by Virosztek, but he  used a different,  more familiar version of the dual cone intersection formula, requiring non-empty interior for one of the cones.

While this condition is fulfilled in the discrete setting, it badly fails in other situations. The reason is clear: In general, $\DD$ has no interior. Indeed, the many functional inequalities---e.g., $f(0)$ being the maximum of the function ---are sensitive to perturbation, and so in a neighborhood of a positive definite function there are many others refuting these properties, hence failing to be positive definite. Therefore, our argument hinges upon finding suitable function spaces and abstract harmonic analysis arguments, fitting together with the possibility to prove the respective dual cone intersection formula in our setup.

In Section \ref{sec:discrete} we explain the method of Virosztek from \cite{virosztek} and extract the main ingredients, in elementary functional-analytic lemmas. This outlines the  proof of the main result of this paper (Theorem \ref{thm:main}) and allows for a quick proof of Theorem \ref{thm:first}, and in an even more general form at that. Section \ref{sec:lcaback} starts with the recollection of some basic facts from the harmonic analysis of LCA groups and describes the function spaces in which we set up the extremal problems. It also  constitutes the identification of the dual space in question. We also discuss the cones that are relevant for us, and determine their duals. Section \ref{sec:Dual-Cone-Intersection} is devoted to the proof of the dual cone intersection formula in the particular setting presented here, which is based on the above mentioned abstract result of Jeyakumar and Wolkowicz \cite{JeyaWolf}. The main results follow in Section \ref{sec:main}, where---after the preparations in the foregoing sections---the proofs will require essentially no effort.

%
%
%
%
%
%
%
%
%
%
%

\section{The functional analysis approach and the proof of Theorem \ref{thm:first}}
\label{sec:discrete}

To illustrate our approach, in this section we present the proof of Theorem \ref{thm:first} in a much more general form.
Suppose $G$ is a discrete Abelian group. Take $X=\ell^1(G;\RR)$, and let $P:=\DD\cap X$ be the cone of positive definite functions in $X$. These are the functions $f$ whose Fourier transform $\widehat{f}$ is a positive function on the dual group $\ft{ G}$ (which is compact). We write $f \gg 0$ to denote positive definiteness. For subsets $S\subseteq G$ we use the abbreviation $S^c:=G\setminus S$.
For $A,B\subseteq G$ we define
\[
Q_{A,B}:=\{f\in X:f|_{A}\leq 0,\: f|_{B}\geq 0\},
\]
which is  a closed convex cone. We identify the dual space $X'$ with $\ell^\infty(G;\RR)$ (the space of real-valued bounded functions  endowed with the supremum norm).
Let $\si\in X' \setminus \{0\}$ be fixed and consider the closed hyperplane $\HH_\si:=\si^{-1}(\{1\})$. For $\rh\in X'$, and $\Omega_+,\Omega_-\subseteq G$ define
\[
 \alpha^\rh_\si(\Omega_+,\Omega_-) :=\inf\{\langle f,\rh\rangle:f\in P\cap Q_{\Omega_+^c,\Omega_-^c}\cap \HH_\si\}
 \]
and
\[
\omega^\rh_\si(\Omega_+^c,\Omega_-^c) :=\sup\{s\in \RR:\rh-s\si\in  P^*+Q_{\Omega_+^c,\Omega_-^c}^*\}.
\]
Here $P^*$ and $Q^*:=Q_{\Omega_+^c,\Omega_-^c}^*$ stand for the dual cones of $P$ and $Q:=Q_{\Omega_+^c,\Omega_-^c}$, respectively. Recall that for  a set $C$ in a real topological vector space $E$ with topological dual  $E'$ we write for the \emph{dual cone}
\[
C^*:=\{\varphi\in E':\langle c,\varphi\rangle\geq0\text{ for every $c\in C$}\}.
\]

\begin{theorem}\label{thm:revvir}
	Let $G$ be a discrete Abelian group, and let $\Omega_+,\Omega_-\subseteq G$ be symmetric sets with $0\in \Omega_+$. Suppose that  $\si$ is strictly positive definite and that $\rh$ is even. Then with the notation and terminology as above we have
\[
 \alpha^\rh_\si(\Omega_+,\Omega_-)=\omega^\rh_\si(\Omega_+^c,\Omega_-^c).
\]
\end{theorem}
In case $G=\ZZ^d$ this result (in a slightly weaker form) is due to R\'ev\'esz   \cite{R-Austral}, with an elegant alternative proof by Virosztek, see \cite{virosztek}.
In what follows we prove Theorem \ref{thm:revvir}, by the method of the latter work, extracting various abstract arguments from the proof.

Let  $E$ be a  real topological vector space with $E'$ its topological dual space, let $C\subseteq E$ be a convex cone, and let $\si,\rh\in E'$. Define the quantities
\[
 \alpha :=\inf\{\langle f,\rh\rangle:f\in C,\, \langle f,\si\rangle=1\}
  \]
  and
   \[
 \widetilde{\omega}:=\sup\{s\in \RR:\rh-s\si\in  C^*\}.
 \]
\begin{lemma}
With the notations from above we  have
 \[
\widetilde{\omega}\leq\alpha.
 \]
\end{lemma}
\begin{proof}
	If either of the sets in the definition of $\alpha$ or $\widetilde{\omega}$ is empty, then the inequality holds trivially with $\inf\emptyset=\infty$ and $\sup\emptyset=-\infty$.
Let $s\in\RR$ such that $\rh-s\si\in C^{*}$, and let $f\in C$ satisfy $\langle f,\si\rangle=1$. Then we have
 \[
 0\leq \langle f,\rh-s\si\rangle =\langle f,\rh\rangle-s\langle f,\si\rangle=\langle f,\rh\rangle-s,
 \]
 so $s\leq \langle f,\rh\rangle$. It follows that $s\leq \alpha$, and thus $ \widetilde{\omega}\leq \alpha$.
\end{proof}

\begin{lemma}\label{lem:2ao}
With the notations from above, if  $\si$ is strictly positive on $C$, i.e., for $f\in C$, $f\neq0$ we have $\langle f,\si\rangle>0$,  then
we have $\alpha\leq\widetilde{\omega}$, so altogether $\alpha=\widetilde{\omega}$.
 \end{lemma}
 \begin{proof}
If $\widetilde{\omega}=\infty$, then there is nothing to prove. So take  $s>\widetilde{\omega}$. It follows that $\rh-s\si\not\in C^{*}$, so there is $f\in C$ with
 \[
 0>\langle f,\rh-s\si\rangle=\langle f,\rh\rangle-s\langle f,\si\rangle,
 \]
in particular $f\neq 0$. By $f\in C$ and by the assumption on $\si$ we may suppose $\langle f,\si\rangle=1$ (positive scalar multiple). Whence we conclude
 $\langle f,\rh\rangle<s$, and  $\alpha<s$.
 Finally $\alpha\leq \widetilde{\omega}$ follows.
 \end{proof}

 The next result, cited from the literature, see, e.g., \cite[Lemma 2.1 (a)]{JeyaWolf}
 \begin{lemma}\label{lem:coneinter1}
	 Let $X$ be a real locally convex space and let $P,Q\subseteq X$ be closed convex cones.
If $0\in P\cap Q$ and the interior of one of the cones intersects the other one, then
		 \[
		 (P\cap Q)^*=P^*+Q^*.
		 \]
\end{lemma}

\begin{proof}[Proof of Theorem \ref{thm:revvir}]
Now we choose in the setting of the previous lemmas $E=X=\ell^{1}(G;\RR)$, $C=P\cap Q$, where $Q=Q_{\Omega_+^c,\Omega_-^c}$, and $\si,\rh\in X'$ as in Theorem \ref{thm:revvir}.
Note that $P^*+Q^*\subseteq (P\cap Q)^*$, so in the setting of the foregoing two lemmas we trivially obtain
	   \begin{align*}
	 \widetilde{\omega}:=\sup\{s\in \RR:\rh-s\si\in  (P\cap Q)^*\}\geq \sup\{s\in \RR:\rh-s\si\in  P^*+Q^*\}=\omega_\si^\rh(\Omega_+^c,\Omega_-^c).
	 \end{align*}
By Lemma \ref{lem:2ao}, as $\si$ is assumed to be  strictly positive definite, we have $\alpha_\si^\rh(\Omega_+,\Omega_-)=\widetilde{\omega}$. We have $0\in P\cap Q$ and, by the assumption $0\in \Omega_+$, also $\car_{\{0\}}\in Q\cap \intt P$, so Lemma \ref{lem:coneinter1} applies. We conclude that
	 $\widetilde{\omega}=\omega_\si^\rh(\Omega_+^c,\Omega_-^c)$, finishing the proof.
	\end{proof}
	
	Theorem \ref{thm:first} is obtained by taking $\Omega_+=\Omega$ and $\Omega_-=G$ in Theorem \ref{thm:revvir}.

\section{Background material from the harmonic analysis on LCA groups}
\label{sec:lcaback}

Here we collect some basic material about  locally compact Abelian (LCA) groups, which will be crucial in our further analysis. We denote by $\haar:=\haar_G$ a Haar measure on such a group $G$. 
For the convolution of two functions $u,v$ on $G$ we write
\[
u\star v(x):=\int_Gu(x-y)v(y)\dd\lambda(y),
\]
whenever this expression exists for (almost) all $x\in G$.

For $\KK\in\{\RR,\CC\}$ we denote by $C(G;\KK)$ the continuous $\KK$-valued functions on $G$, and write $C_c(G;\KK)$ for the subspace of functions with compact support. For $p\in [1,\infty]$ the Lebesgue spaces of $\KK$-valued, $p$-integrable  (resp.{} essentially bounded) functions with respect to the fixed Haar measure are denoted by $L^p(G;\KK)$.

\subsection{Lattices}
We start with recalling the following  structural result.

 \begin{proposition}\label{prop:compact tile} For a compactly generated LCA group $G$ the following hold.

There is a  discrete subgroup $L$ in $G$ which is isomorphic to $\ZZ^d$ and there is a relatively compact Borel set $B\subseteq G$ such that $G/L$ is compact, $B$ tiles with complement $L$, i.e., $G=L+B$ with each $g\in G$ represented uniquely as $g=\ell+b$ with $\ell\in L$ and $b\in B$.

Moreover, the set $B$ can be chosen such that its closure $\overline{B}$ is symmetric, i.e., $\overline{B}=-\overline{B}$, and its boundary has Haar measure zero: $\haar_G(\partial B)= \haar_G(\overline{B} \setminus \intt B)=0$.
\end{proposition}
\begin{proof}This is directly seen from the structure theorem of compactly generated LCA groups, and is actually an ingredient of the proof of this central result; e.g. \cite[2.4.2 Lemma]{rudin}.
\end{proof}

From now on, as has been said in the introduction, we shall assume that the LCA group $G$ is compactly generated and fix a \emph{tile} $B$ and a  discrete subgroup $L$ (called a  \emph{lattice}) as provided by the previous proposition. We remark, again, that the assumption that $G$ is compactly generated is only for convenience, and is not needed for the main result, Theorem \ref{thm:mainmain}, see also Remark \ref{rem:mainmain}.

As $L$ is a discrete subgroup,  tiling $G$, there is a set of generators $\{g_1,\ldots,g_d\}$ such that each $\ell\in L$ is represented uniquely as $\ell=n_1g_1+\dots+n_d g_d$ with $n_j\in\ZZ$. With this unique representation we write $\|\ell\|_L:=\max_{j=1,\ldots,d} |n_j|$. Note that for a fixed  relatively compact $V$ also $B+V$ is relatively compact, so there is a number $m$ such that $V \cap (\overline{B}+\ell) = \emptyset$ if some coefficient $n_j$ of $\ell$ is at least as large as $m$, that is, if $\|\ell\|_L \ge m$.

Denoting the inversion by $\inv:G\to G$, $g\mapsto -g$, for a function $u:G\to \KK$ we set $\widetilde u:=\overline{u\circ \inv}$. If $u\in L^2(G)$, then the convolution $u \star \widetilde u$ defines a  continuous, positive definite function.
The following result, a variant of the existence of \textbf{Fejér's kernel}, is standard, we include its proof only for the convenience of the reader.
\begin{lemma}\label{l:posdefkernel}
	For every relatively compact set $K \subseteq  G$ and every $\ve>0$ there exists a symmetric, relatively compact Borel set $H \subseteq G$ such that for the characteristic function $\one_H$ of $H$ the function   $k:=\frac1{\haar(H)}\one_H \star \one_H$ has  the properties
\begin{enumerate}[(1)]
\item $k(0)=1$,
\item $k \ge 0$,
\item $k \in C_c(G;\RR)$,
\item $k \gg 0$,
\item $k|_K \ge 1-\ve$.
\end{enumerate}
\end{lemma}
\begin{proof}
	First, let $H$ be an arbitrary symmetric Borel set with compact closure and abbreviate $h:=\one_H$, so that $\haar(H)<\infty$,  $h \in L^2(G;\RR)$, and properties (1) and (2) follow immediately (irrespective of the choice of $H$). Also, in view of $h \in L^2(G;\RR)$, we have $h\star h \in C(G;\RR)$, with $h\star h|_{G\setminus(H+H)}=0$, so $k \in  C_c(G;\RR)$  and (3) is proved. Moreover, if $H$ is symmetric, we have $h=h\circ\inv$, so $h=\widetilde h$ (as $h$ is real-valued), and $k$ is therefore positive definite, i.e., (4) holds.

The only property to be ensured by (an appropriately large) choice of the symmetric Borel set $H$ is the last one.  We may assume that $K$ is compact, then also $K-\overline{B}$ is compact, and by discreteness of $L$ the set $L(K):=(K-B)\cap L = \{\ell \in L :  K \cap (B+ \ell ) \ne \emptyset \}$ is finite.

Let $m_K:=\max_{\ell \in L(K)} \|\ell\|_L$. We define similarly $L(B)$ and $m_B$ and put $M:=m_K+m_B$.
 Then $K \subseteq  Q_n:=\cup \{ B+\ell : \|\ell\|_L \le n\}$ for all $n \ge m_K$. We will choose $H:=\overline{Q_n}$ with some sufficiently large $n$. Obviously, together with $\overline{B}$ also $H$ is symmetric, and $H$ is a compact set as needed.

Let now $x \in K$ be arbitrary and consider
\begin{align*} h\star h (x) & = \int_G h(x-y) h(y) \dd  \lambda(y) = \int_H h(x-y) \dd \lambda(y)
\\ & \ge \int_{Q_{n-M}} \dd \lambda(y) = \sum_{\ell, \|\ell\|_L\le n-M} \int_{B+\ell} \dd y = (2(n-M)+1)^d \haar(B).
\end{align*}
Here in the step getting $\ge$ we used that if $y \in B+\ell_y$ with $\|\ell_y\|_L \le n-M$ and $x \in B+\ell_x$ with $\ell_x \in L(K)$ (and therefore $\|\ell_x\|_L\le m_K$), then $x-y \in B-B+\ell_x-\ell_y$, so that $x-y \in B+\ell$ with $\| \ell \|_L \le \|\ell_x\|_L +\|\ell_y\|_L+m_B \le n $.

On the other hand, we have $\haar(H)=(2n+1)^d \haar(B)$. It follows that for arbitrary $x \in K$ we have
\[
k(x) \ge \frac{(2(n-M)+1)^d\haar(B)}{\haar(H)}=\frac{(2(n-M)+1)^d }{(2n+1)^d} >1-\ve\quad\text{if $n$ is large enough}.
\]
Property (5), hence also the lemma, is established.
\end{proof}

 \subsection{Spaces of functions and functionals}

The space over which we shall formulate the extremal problem is a particular kind of so-called \textbf{amalgam spaces}, first considered by Wiener in his papers \cite{Wiener, Wiener2} about convergence of Fourier series, and later further developed by many authors, see, e.g., Argabright and Gil de Lamadrid \cite{AG-apm}, Holland \cite{Holland1975, Holland75b}, Stewart \cite{Stewart1979}, Feichtinger \cite{Feichtinger1977, Feichtinger1983}, Fournier, Stewart \cite{FournierStewart1985}.

The amalgam space considered here is first studied Wiener in \cite{Wiener2} for $G=\mathbb{R}$ and in the general situation by Phuong-C\'ac in \cite{Cac}, and later more or less simultaneously by  Holland \cite{Holland1975}, Stewart \cite{Stewart1979},  Bertrandias, Dupuis \cite{BD}.

 An excellent overview of amalgam spaces in context of LCA groups and their use appears in Fournier and Stewart \cite{FournierStewart1985}, where the authors additionally refer to the Mathematical Reviews article \cite{GilReview} by Gil de Lamadrid as an informative account of the historical development.

One defines
\[
X:=C^{\infty,1}(G;\RR):=\Bigl\{ f \in C(G;\RR)~:~ \|f\|_X:=\sum_{\ell \in L} \|f|_{B+\ell}\|_\infty < \infty\Bigr\}.
\]
The space\footnote{In the literature another notation is $(C,\ell^1).$} $X$ is also of ``mixed-norm'' type, and is a Banach space with the given norm.

  It is easy to see that the set $C^{\infty,1}(G;\RR)$ and the topology thereon, given by the above defined norm, is independent of the particular choice of the lattice and the tile.
Similarly, one defines the mixed-norm space $C^{\infty,1}(G;\CC)$ of complex-valued functions.

\begin{remark}
	\begin{enumerate}[(a)]
\item One has $C^{\infty,1}(G;\RR)\subseteq C(G;\RR)\cap L^\infty(G;\RR)\cap L^1(G;\RR)$.
	\item Note that $C^{\infty,1}(G;\RR)$ is not just a Banach space but even a Banach lattice with the pointwise ordering.  Moreover, the space $C_c(G;\RR)$ of compactly supported, real-valued, continuous functions is contained in $C^{\infty,1}(G;\RR)$ and dense (for this, along with some refined statements, see Proposition \ref{prop:density} below). (The analogous statements hold for the complex-valued spaces as well.)
	\item
An important feature of this setting is that for (finitely generated and) discrete groups we obtain $X=\ell^1(G;\RR)$ with equivalent norms; while for compact $G$ we have $X=C(G;\RR)$ with equivalent norms.
\end{enumerate}
\end{remark}

In the following we will denote the set of continuous (possibly complex-valued) positive definite functions as $\DD$, and set $\Dc:=\DD\cap C_c(G;\CC)$.

\begin{proposition}\label{prop:density} $C_c(G;\RR)$ is dense in $C^{\infty,1}(G;\RR)$ and also $\Dc$ is dense in $\DD\cap C^{\infty,1}(G;\RR)$.
\end{proposition}

Note that the statements are well-known for the usual topologies of uniform or compact convergence, but here we deal with a mixed-norm space with the larger norm $\| \cdot \|_X$ than  $\| \cdot \|_\infty$.
\begin{proof} Take a function $f\in X=C^{\infty,1}(G;\RR)$, and consider the norm represented by the sum  $\|f\|_X=\sum_{\ell \in L} c_\ell$, where $c_\ell:=\|f|_{B+\ell}\|_\infty$. As the sum is convergent, for any $\ve>0$ there is a finite set $L' \subseteq  L$ with $\sum_{\ell \in L\setminus L'} c_\ell <\ve$. Take $K:=B+L'$. This is a relatively compact set, so that Lemma \ref{l:posdefkernel} can be applied and we find a positive definite kernel $k$ with the properties (1)-(5) listed there. Now if $\ell \in L$, then
\[
\| (f - k\cdot f)|_{B+\ell} \|_\infty \le \|(1-k)|_{B+\ell}\|_\infty  \|f|_{B+\ell} \|_\infty \le
\begin{cases}  \ve c_\ell, \qquad & \text{if} \quad \ell \in L',
\\
c_\ell, \qquad & \text{if} \quad \ell \in L\setminus L'.
\end{cases}
\]
Altogether,
\[
\| f- k\cdot f\|_X \le \sum_{\ell \in L'} \ve c_\ell  + \sum_{\ell \in L\setminus L'} c_\ell \le \ve \|f\|_X +\ve.
\]
This proves $k \cdot f \in X$ and the first approximation statement, for $k \cdot f \in C_c(G;\RR)$. For the last statement it suffices to recall that $k$ from Lemma \ref{l:posdefkernel} was also positive definite, and for a positive definite $f$ also $k \cdot f$ is such, so that $k\cdot f \in \Dc \cap X$, too.
\end{proof}

The measure theoretic counterpart of the space $X$ is also of amalgam type, and was introduced by Phuong-C\'ac in \cite{Cac}, and studied later by  Liu, van Rooij, and Wang \cite{lrj},  Holland \cite{Holland1975, Holland75b},  Bertrandias, Datry, Dupuis \cite{BDD}, Stewart \cite{Stewart1979}.
We will concentrate on a special case, that is the space of translation bounded measures, see the papers  Argabright, Gil de Lamadrid \cite{AG-Memoirs}, Lin \cite{Lin}, Thornett \cite{Thornett} and the monograph Berg, Forst \cite{BergForst}.
We collect some details here for later reference.

One natural topology on $C_c(G;\RR)$ is the inductive limit topology defined by the subspaces $\{f\in C(G;\RR):\supp(f)\subseteq K\}$ for  $K\subseteq G$ compact. Elements of the dual space $(C_c(G;\RR))'$ are called \emph{Radon measures}, see \cite{AG-Memoirs}. If $\psi \in(C_c(G;\RR))'$ is a positive functional, then the Riesz-Markov-Kakutani theorem yields the existence  of a positive $\si$-additive measure $\mu$ on the Borel $\sigma$-algebra such that
\[
\psi(f)=\int_Gf\dd\mu\quad\text{for all $f\in C_c(G;\RR)$},
\]
and uniqueness is also ensured if one requires regularity properties from $\mu$, making it a \emph{Radon measure in the sense of measure theory}. For unambiguity let us call such a $\mu$ a Radon-Borel measure, to emphasize that we indeed have a ($\si$-additive) measure on the Borel $\si$-algebra. Any functional  $\psi \in(C_c(G;\RR))'$ can be represented uniquely as follows: There are $\mu^+,\mu^-$ positive Radon-Borel measures such that
\[
\psi(f)=\int_Gf\dd\mu^+-\int_Gf\dd\mu^-\quad\text{for all $f\in C_c(G;\RR)$}.
\]
Note however that $\mu^+-\mu^-$ does not necessarily exist as a signed measure, but $|\mu|:=\mu^++\mu^-$ is  a positive Radon-Borel measure and represents the functional $|\psi|\in (C_c(G;\RR))'$.

 \medskip From the above arguments, we immediately obtain that for every linear functional $\psi\in X'$ there are positive Radon-Borel measures  $\mu^+,\mu^-$ such that
\[
\psi(f)=\int_Gf\dd\mu^+-\int_Gf\dd\mu^-\quad\text{and}\quad |\psi|(f)=\int_Gf\dd|\mu|\quad \text{for all $f\in X$}.
\]
From now on, we shall not distinguish a positive Radon measure from its representing Radon-Borel measure.


	A Radon measure $\psi$ is called \emph{translation bounded} if for some  compact neighborhood~$K$ of $0$ there exists a  constant $C$ such that
	\begin{equation}\label{eq:TB}
	\sup_{x\in G}|\psi|(K+x)\le C.
	\end{equation}
(See  Argabright, Gil de Lamadrid \cite{AG-Memoirs}, Lin \cite{Lin}, Thornett \cite{Thornett} and Berg, Forst \cite{BergForst}.)
The space of translation bounded (real-valued) Radon measures is denoted\footnote{Another often used notation is $M_\infty$.} by $M:=M(G;\RR)$.
It is easy to see that if $\psi$ is translation bounded, then \eqref{eq:TB} holds for every (relatively)  compact set $K\subseteq G$.
We introduce the norm
\[
\|\psi\|_M:=\sup_{\ell \in L} |\psi|(B+\ell)
\]
turning $M(G;\RR)$ into a Banach space. Note that we could have chosen as well the norm $\|\psi\|^* := \sup_{x\in G} |\psi|(B+x)$, which is fitting to \eqref{eq:TB} with $K=B$ and defines the same topology.

Given $f\in C^{\infty,1}(G;\RR)$ and $\psi\in M(G;\RR)$ we have that
\[
\langle f,\psi\rangle:= \sum_{\ell \in L}\Bigl( \int_{B+\ell} f \dd\psi^+-\int_{B+\ell} f \dd\psi^-\Bigr)\in \RR,
\]
and
\[
|\langle f,\psi\rangle|\leq\sum_{\ell\in L}\|f|_{B+\ell}\|_\infty \cdot |\psi|(B+\ell)\leq \|\psi\|_{M}\cdot \|f\|_{X}.
\]

It follows that $\psi\in X'$ with $\|\psi\|_{X'}\leq \|\psi\|_{M}$.
The following result, describing the dual space of $X$,  is well-known, see, e.g., Goldberg \cite{Goldberg1967}, Phuong-C\'ac \cite{Cac}, Feichtinger \cite{Feichtinger1977}, Liu, van Rooij, Wang \cite{lrj} or Stewart \cite{Stewart1979}. We include the proof for illustration, and because it is particularly short.
\begin{lemma}\label{l:Dual-of-X} The dual space of the Banach space $C^{\infty,1}(G;\RR)$ is (isomorphic to) $M(G;\RR)$.
\end{lemma}
\begin{proof} We saw above that translation bounded Radon measures give rise to continuous linear functionals on $X=C^{\infty,1}(G;\RR)$. Conversely, take $\psi\in (C^{\infty,1}(G;\RR))'$; we need to prove that $\psi$ is translation bounded. By decomposing into positive and negative parts, we may suppose without loss of generality that $\psi$ is a positive functional. Let $(\ell_k)\subseteq L$ be such that $\psi(B+\ell_k)\to \sup_{\ell \in L} \psi(B+\ell)$ as $k\to\infty$. Take an arbitrary relatively compact, open set $V$ containing $B$. Then there are at most finitely many, say $N\in \NN$, lattice points $\ell\in L$ such that $(B+\ell) \cap V\neq \emptyset$ (since $L$ is discrete and $V-B$ is relatively compact). Let $\veps>0$ be arbitrarily given and let $k\in \NN$ be fixed. Take a compact $K_k\subseteq B+\ell_k$ and  a relatively compact, open $U_k\supseteq B+\ell_k$ with  $\psi(U_k\setminus K_k)<\veps$.
We may suppose also that $U_k\subseteq V+\ell_k$, so that $U_k$ intersects at most $N$ cells of the form $B+\ell$ ($\ell\in L$).  Take a continuous function $f_k$ with support $\supp ( f_k)\subseteq U_k$, $f_k|_{K_k}=1$, $f(G)\subseteq [0,1]$. Then $f_k\in C_c(G;\RR)$ and $\|f_k\|_X\leq N$, so
	\begin{align*}
	\|\psi\|_{X'}N&\geq \|\psi\|_{X'}\cdot\|f_k\|_X\geq \psi(f_k)=\int_G f_k\dd\psi
	=\int_{U_k} f_k\dd\psi
	\geq \int_{K_k} f_k\dd\psi\\
	&=\psi(K_k)=\psi(B+\ell_k)-\psi((B+\ell_k)\setminus K_k)\geq \psi(B+\ell_k)-\veps.
	\end{align*}
With $k\to \infty$ we conclude
	\[
	N\|\psi\|_{X'}\geq \sup_{\ell \in L} \psi(B+\ell)-\veps = \| \psi \|_M - \veps,
	\]
	and hence $\psi$ is translation bounded.	

The equivalence of the norms $\|\cdot\|_{X'}$ and $\|\cdot\|_M$ follows at once, too.
\end{proof}

\label{sec:dualcones}
\begin{proposition}\label{prop:Mdual}
For a fixed subset $A\subseteq G$ consider the convex cone
	\[
	Q_A:=\{f\in C^{\infty,1}(G;\RR):f|_A\leq 0\}.
	\]
Then for the dual cone we have
	\begin{equation}\label{eq:MA}
	Q_A^*=\{\psi\in M(G; \RR):\psi=-\psi^-\text{ and $\psi^-(G\setminus \oA)=0$}\}.
	\end{equation}
\end{proposition}
\begin{proof}
If $A\subseteq  G$ is an arbitrary set, then every continuous function $f\in C(G;\RR)$ with $f|_A \le 0$ satisfies also $f|_{\overline{A}}\le 0$.
Therefore, $Q_A=Q_{\oA}$ and henceforth it suffices to deal with closed sets, whereas for arbitrary sets we can apply the resulting description of the dual cone for $\overline{A}$ in place of $A$.

If $\psi$ is in the set on the right-hand side in \eqref{eq:MA} and $f\in C_c(G;\RR)\cap Q_A$, then	
\[
\psi(f)=-\psi^-(f)=-\int_G f\dd\psi^-=-\int_A f\dd\psi^-\geq 0.
\]
So by denseness\footnote{We use that $C_c(G;\RR)$ is dense in $X$ with respect to the topology of $X$. See Proposition \ref{prop:density}.} we obtain $\psi\in Q_A^*$.

\medskip\noindent
Conversely, suppose that $\psi\in 	Q_A^*$. We first prove that $|\psi|(G\setminus A)=0$.
Take an arbitrary compact set $K\subseteq G\setminus A$ and  for $\veps>0$ take a relatively compact, open  set $U$ with $K\subseteq U\subseteq G\setminus A$ and $|\psi|(U\setminus K)<\veps$. Consider the Hahn-decomposition of the signed Radon-Borel measure $\psi|_U=\psi^+|_U-\psi^-|_U$: we can write $U=U^+\cup U^-$ with $U^+\cap U^-=\emptyset$, $\psi^+(U^-)=0=\psi^-(U^+)$. Let $K^{\pm} \subseteq K \cap U^{\pm}$ be compact sets with $|\psi|(K\setminus (K^+\cup K^-))<\veps$. Finally, using Urysohn's Lemma, take a continuous function $f$ with $\supp (f)\subseteq U$, $f|_{K^+}=-1$, $f|_{K^-}=1$ and $f(G)\subseteq [-1,1]$. Then $f\in C_c(G;\RR)$ and $f|_{A} \equiv 0$, so $f\in Q_A$ and we have
\begin{align*}
0\leq \psi(f)&=\int_U f\dd\psi \leq 2 \veps+\int_{K^+ \cup K^-} f\dd\psi = 2\veps+\int_{K^+} f\dd\psi^+-\int_{K^-}f \dd\psi^-\\
&= 2\varepsilon-\psi^+(K^+)-\psi^{-}(K^-) = 2\varepsilon-|\psi|(K^+\cup K^-)\leq 3\varepsilon-|\psi|(K).
\end{align*}

\noindent This being true for all $\veps >0$ implies $|\psi|(K)=0$, and by regularity $|\psi|(G\setminus A)=0$ follows.

\medskip\noindent Next we prove $\psi^+(A)=0$. Suppose the contrary, i.e., $\psi^+(A)>0$, so that there is a compact $H\subseteq A$ with $\psi^{+}(H)>0$ and $\psi^{-}(H)=0$. For $\veps\in (0,\psi^+(H))$ take $V$ relatively compact open with $H\subseteq V$ and $|\psi|(V\setminus H)<\veps$ and a continuous function $g$ with $\supp(g)\subseteq V$, $g|_H=-1$ and $g(G)\subseteq [-1,0]$. Then $g\in C_c(G;\RR)$ and $g \in Q_A$ as $g\le 0$ everywhere, so using $\psi \in Q_A^*$ we find
\begin{align*}
\psi(g)&=\int_V g\dd\psi^+-\int_Vg \dd\psi^-\leq \veps +\int_H g\dd\psi^+-\int_H g \dd\psi^-=\veps+\int_H g\dd\psi^+=\veps-\psi^+(H)<0.
\end{align*}
It follows that  $\psi\not \in Q_A^*$, a contradiction. We thus conclude  $\psi^+(A)=0$, so $\psi^+=0$. The proof is complete.
\end{proof}

An \emph{even} function $f:G\to \RR$ is one that is invariant under inversion $\inv:G\to G$, $g\mapsto -g$, i.e., that satisfies $f\circ \inv=f$; while $f$ is \emph{odd} if $f\circ \inv=-f$.
Dually a Radon measure $\psi$ is even, respectively odd, if it satisfies
$\psi(f\circ \inv)=\psi(f)$, respectively, $\psi(f\circ \inv)=-\psi(f)$ for every $f\in C_c(G;\RR)$. The family of real-valued and translation bounded, odd Radon measures will be denoted as $\OO$. Obviously,  $\OO$ is a closed subspace, hence is also a closed cone of $M(G;\RR)$.
\begin{proposition}[R\'ev\'esz-Ga\'al \cite{Zeitschrift}]\label{prop:realposdeffunct}
A function $f\in C(G;\RR)$ is positive definite if and only if it is even and  \emph{positive definite in the real sense}, the latter property meaning that
\[
\sum_{j=1}^n \sum_{k=1}^n c_j c_k f(x_j-x_k) \ge 0 \quad\text{for all $c_1,\ldots,c_n \in \RR$, $n \in \NN$}.
\]
\end{proposition}

Note that here the usual requirement of non-negativity of quadratic forms is required only with real (as opposed to the usual complex) coefficients, hence no conjugation is used either. To do real linear programming it is necessary to restrict to $C(G;\RR)$, and the description of the class is then needed throughout.

\begin{definition}
	A Radon measure (that is, a linear functional) $\mu \in M(G;\CC)$ is called a \emph{measure of positive type}, if for all ''weight functions'' $ u\in C_c(G;\CC)$ we have
	\begin{equation}\label{eqposdef}
	\langle u\star \widetilde{u},\mu\rangle \ge 0.
\end{equation}
We denote this by writing $\mu\gg 0$.
The family of all real-valued measures of positive type will be denoted by $\MM$.
\end{definition}

\begin{proposition} A measure $\mu \in M(G;\RR)$  belongs to $\MM$ if and only if it is even and satisfies \eqref{eqposdef} for all \emph{real-valued} weights $u \in C_c(G;\RR)$. 
\end{proposition}

Note that the point of the above is that in case we only assume \eqref{eqposdef} for  real weight functions, then we do not get evenness of the measure (as is known for (integrally) positive definiteness). The same way, if a continuous function satisfies \eqref{eqposdef2} with only real coefficients, then it may not be even. Altogether, the classes satisfying the defining equations only for real weights or coefficients is larger than the family of the real-valued elements of the respective positive definite classes. Note that, e.g., an odd measure is orthogonal to any---necessarily even---function $u \star \widetilde{u}$ with $u \in C_c(G;\RR)$, hence it satisfies $\langle u\star \widetilde{u},\mu\rangle\ge 0$ for real weights. According to the proposition, the class of functions with ``positive definiteness with respect to real weight functions'' can be obtained as the sum of the sets of the truly positive definite (and real-valued) ones and the odd ones.

\begin{proof} ``$\Rightarrow$'': The part on being real-valued, and even non-negative on $u\star\widetilde{u}$ for $u \in C_c(G;\RR)$ is clear, as a $\mu \in \MM$ is by definition real-valued, and is a measure of positive type with respect to real-valued weights, too.

It remains to see why it is even. Key to this is the property\footnote{By definition $\widetilde{\mu}(f):=\overline{\mu(\widetilde{f})}$, compare Argabright, Gil de Lamadrid \cite{AG-Memoirs}, page 2, where the $~\widetilde{}~$ operation is called ``the involution'' and is denoted by a $~^\star~$. } $\mu=\widetilde{\mu}$ for $\mu$ positive definite. The symmetry property $\mu=\widetilde{\mu}$ is given in Proposition 4.1 of \cite{AG-Memoirs}, compare the definition of the involution $^\star$ on page 2 and the definition of the family of positive definite Radon measures on page 23.

Let $u \in C_c(G,\RR)$. Then $\langle u,\mu\rangle=\langle u,\wt{\mu}\rangle=\overline{\langle \wt{u}, \mu\rangle} =\langle u\circ\inv ,\mu\rangle$, for now both $\mu$ and $u$ are real-valued. This is equivalent to $\mu$ being even, as one can approximate the characteristic function of any compact set by such $u \in C_c(G;\RR)$ arbitrarily well even in $X$-norm.

\medskip\noindent  ``$\Leftarrow$'': Let $\mu\in M(G;\RR)$ be even and non-negative on convolution squares of real-valued $ u \in C_c(G;\RR)$. Take a complex-valued weight $w=u+iv$. Then we have
\[
\langle w\star\widetilde{w},\mu\rangle=\langle(u+iv) \star (\wt{u} - i\wt{v}),\mu\rangle = \langle(u \star \wt{u} + v \star \wt{v} + i(v \star \wt{u}-u \star \wt{v} ) ),\mu\rangle.
\]
Here $\langle u \star \wt{u},\mu\rangle \ge 0$ and $ \langle v \star \wt{v},\mu\rangle$ are non-negative by assumption, since $u, v \in C_c(G;\RR)$. Further, it is easy to check that $h:=v \star \wt{u}-u \star \wt{v} $ is an odd function, so that in case $\mu$ is even, we necessarily have $\langle h,\mu\rangle=0$. This proves $\langle w\star\widetilde{w},\mu\rangle\ge 0$, hence the assertion.
\end{proof}

\begin{proposition}\label{prop:Pdual}
The dual cone of $P:=\DX:=\DD\cap C^{\infty,1}(G;\RR)$ is
\[
P^*=\mathcal{M}+\mathcal{O}.
\]
\end{proposition}
\begin{proof}
First we prove $\OO \subseteq  P^*$. Since a function $f\in P$ is even, for an odd measure $\psi\in \mathcal{O}$  we have
\begin{align*}
-\psi(f\circ \inv)=\psi(f)=\psi(f\circ \inv),
\end{align*}
so that $\psi(f)=0$ follows, i.e., $\mathcal{O}\subseteq P^*$.

Next we prove $\MM \subseteq  P^*$. Let $\psi\in\mathcal{M}$ be a measure of positive type in $M(G;\RR)$, so that in particular it is real-valued and translation bounded. We want to get $\psi(f)\ge 0$ for all $f \in P$. This is essentially Theorem 4.3 of Argabright, Gil de Lamadrid in \cite{AG-Memoirs} which says that any positive definite measure $\psi$ is non-negative on all continuous positive definite functions $f$ (not just on the ones with a representation as a convolution square of $C_c(G;\RR)$ functions), provided $f$ is integrable with respect to $\psi$. Note that for $f \in X$ we have this integrability condition once $\psi \in M(G;\RR)$ is translation bounded. In other words, $\psi(f)\ge 0$ for all $f \in P$, so $\psi$ is in $P^*$, which is exactly what we want.

Up to here we have $\OO, \MM \subseteq  P^*$, and as  $P^*$ is a cone, also $\MM+\OO \subseteq  P^*$.

\medskip
Conversely, we want to show\footnote{Note that this is in principle the harder direction, as usually from $K_1, K_2 \subseteq  C$ with certain cones $K_1, K_2, C$ it follows only that $K_1+K_2 \subseteq  C$---and therefore also $\overline{K_1+K_2} \subseteq  C$, if $C$ is closed---but the converse requires an argument.} that $P^* \subseteq  \MM+\OO$. Take any $\psi\in P^*$.
We then need to see $\psi \in \OO + \MM$.

Let us decompose $\psi$ to an odd and even part: $\psi=\psi_o+\psi_e$, $\psi_e=(\psi+\wt\psi)/2$, $\psi_o=(\psi-\wt\psi)/2$. Obviously, $\psi_o \in \OO$. So, it remains to see that $\psi_e \in \MM$, i.e., $\psi_e$ is a measure of positive type. This means $\int_G u\star \wt u\dd\psi_e=\langle u \star \widetilde{u},\psi_e \rangle \ge 0$ for all $u \in C_c(G;\CC)$. Note that here, exceptionally, we need to deal with complex-valued functions $u$, too.

As $u \star \widetilde{u}$ is (complex-valued and) positive definite, and it is also compactly supported, its real part $f:= \Re (u \star \widetilde{u})$ is also positive definite, real-valued and compactly supported, hence belongs to $P$. Now by condition $0 \le \langle f, \psi\rangle = \langle \Re (u\star\widetilde{u}),\psi   \rangle = \Re \langle  u \star \widetilde{u},\psi \rangle$, given that $\psi$ is a real-valued Radon measure.

Let us write $g:=\Im (u\star \wt u)$. Then with $F:=u \star \widetilde{u}=f+ig$ we have  $\widetilde{F}=F$, so $f(x)+ig(x)=F(x)=\overline{F(-x)}=f(-x)-ig(-x)$. Therefore, $f$ is even and $g$ is odd. It follows that $\langle u \star \widetilde{u},\psi \rangle=\langle F,\psi_e \rangle + \langle F,\psi_o\rangle =\langle  f , \psi_e\rangle + i\langle g,\psi_o \rangle$ and $\Re \langle u \star \widetilde{u},\psi \rangle=\langle  f ,\psi_e\rangle$. So, $\langle f,\psi_e \rangle = \Re \langle u \star \widetilde{u},\psi \rangle = \langle  \Re (u \star \widetilde{u}), \psi\rangle \ge 0$, and $\psi_e$ is a positive definite Radon measure $\psi_e \in \MM$. The proof of $P^* \subseteq  \OO + \MM$ is complete.
\end{proof}


\section{A dual cone intersection formula}\label{sec:Dual-Cone-Intersection}


%
Recall the definition of the Banach space $X=C^{\infty,1}(G;\RR)$.
We set $P=X\cap \DD$, i.e., $P$ is the closed cone of positive definite functions in $X$. For an arbitrary set $A$ we define
\[
Q:=Q_A:=\{f\in X:f|_A\leq 0\}.
\]
In Section \ref{sec:dualcones} we have determined the dual cones.
As seen in Section \ref{sec:discrete} the crux of the whole argument for the strong duality result is lies  in proving $(P\cap Q)^*=P^*+Q^*$, without the need for weak$^*$ closure here. We want to invoke the Jeyakumar-Wolkowicz Lemma---Lemma 2.2 (a) in \cite{JeyaWolf}---which tells that a sufficient condition is that $P-Q$ is a closed subspace. We will succeed, similarly to Ga\'al, R\'ev\'esz \cite{Zeitschrift}, with the following.

\begin{theorem}\label{t:Ppluslemma} With the previous notation we  have $P-Q=X$.
\end{theorem}
 For arbitrary $f\in X$ we need to  find $p \in P$ and $m \in Q$  such that $f=p-m$, or $p=f+m$. In other words, we need $p\gg 0$ in $X$ such that $p|_A \le f|_A$.
This is non-trivial, it needs a construction with triangle functions etc. We start with an auxiliary result.

\begin{lemma}[Sign swap lemma]\label{l:signswap}
	 Let $S \subseteq  G$ be a symmetric Borel set of positive Haar measure  with compact closure and $0 \not\in \overline{S}$.
Then there exists a positive definite $k \in C_c(G;\RR)$ such that $k|_S \equiv -1$.
Furthermore, given any open neighborhood $V$ of $0$ with compact closure, such that with $W:=V-V$ it holds $(W+W+W) \cap S = \emptyset$, the function~$k$ can be given with the following properties.
\begin{enumerate}[(i)]
  \item If $k(x)>0$, then $x \in W$. That is, $k_{+}$ ``lives'' on $W$.
  \item If $k(x)<0$, then $x \in S+W+W$. That is, $k_{-}$ ``lives'' on $S+W+W$.
  \item $\int_G k\dd\haar =0$ and $\frac12 \| k\|_1 = \int_G k_{+}\dd\haar = \int_G k_{-}\dd\haar = \lambda(S+W)$.
  \item $k(0)=2 \lambda(S+W)/\lambda(V)$.
  \end{enumerate}
\end{lemma}
\begin{proof} Denote the ``generalized triangle function'' as $r:=r_V:=\frac{1}{\lambda(V)} \car_V \star \car_{-V}$ and its translate by $x$ as $T_xr$, that is
\[
T_xr:=r(\cdot-x).
\]
Consider any $x \in S+ W$. Then $T_xr$ vanishes outside $S+W+W$. It is obvious that $r$ vanishes outside $W$, hence we have $r \cdot T_xr\equiv 0$, (i.e., they attain non-zero values on disjoint sets, one being a subset of $W$, and the other one a subset of $S+W+W$). It is equally easy to see that
\[
r_x := 2r-T_x r - T_{-x} r \gg 0.
\]
Indeed, $r_x = r \star (2\delta_0-\delta_x-\delta_{-x})$, where $r$ is a positive definite function as a convolution square, and the measure on the right-hand side of the convolution is an (integrally) positive definite one, too.

We now define
\[
g(y):=\int_{S+W} r_x(y) \dd \lambda(x) = \int_{S+W} (2r(y)-r(y-x)-r(y+x)) \dd\lambda(x).
\]
Let now $y \in S$. Then we also have
\[
g(y) = - \int_{S+W} (r(y-x) +r(y+x)) \dd \lambda(x) = - 2 \int_{S+W} r(y-x) \dd \lambda(x),
\]
because $r \equiv 0$ on $S+W$ and because $S$ and $S+W$ are symmetric sets. Observe that the set where $r(y-\cdot) \ne 0$ is a subset of $-W + y \subseteq  -W+S = W+S$. Thus for $y \in S$
\[
g(y)=- 2 \int_{S+W} r( y-x) \dd \lambda(x) = -2 \int_{r(y-\cdot)\ne 0} r( y-x)\dd \lambda(x) = -2 \int_G r\dd\haar = -2 \lambda(V).
\]
Therefore, $k:=\frac{1}{2\lambda(V)} g$ is constant $-1$ on $S$. Furthermore, by construction $g$, and hence $k$, is also positive definite: as each $r_x \gg 0$, their sum, or integral also satisfies the defining equations of positive definiteness.

The property (i) and (ii) are guaranteed by construction. Indeed, if $y \not\in W$, then $r(y)=0$ and the defining integral of $g$ contains only negative terms; and similarly, if $y \not\in S+W+W$, then for any $x \in S+W$ we have $y \pm x \not\in W$, and $r(y \pm x) =0$, showing that the defining integral of $g$ can have only non-negative terms. Observe that this consideration also yields that there is no $y$ with the appearance in the defining integral of $g$ both positive and negative terms; if there appears a positive value of $r(y)$, then $y \in W$ and $r(y \pm x)= 0$, for all $x \in S+W$, and if there appears a negative value $-r(y-x)<0$ then $y \in S+W+W$ and $r(y)=0$. Therefore, $g_{+}(y)=\int_{S+W} 2r(y) \dd \lambda(x) = 2 \lambda(S+W) r(y)$ and (i) follows, and $g_{-}(y)=-2 \int_{S+W} r(y-x) \dd \lambda(x)$, and also (ii) is obtained.

As each $\int r_x\dd\haar =0$, also (iii) is immediate. Finally,
\[
\frac12 \|k\|_1 = \int_G k_{+} \dd\haar= \frac{1}{2\lambda(V)} \int_G g_{+}\dd\haar = \frac{1}{2\lambda(V)} \int_G 2 \lambda(S+W) r\dd\haar = \lambda(S+W) .
\]
and
\[
k(0) = \frac{1}{2\lambda(V)} g(0) = \frac{\lambda(S+W)}{\lambda(V)}.
\]
\end{proof}

\begin{proof}[Proof of Theorem  \ref{t:Ppluslemma}.]
	Let  $f\in C^{\infty,1}(G)$ be arbitrary and let $c_\ell:=\|f|_{B+\ell}\|_\infty$ ($\ell \in L$).

As $\intt B \ne \emptyset$, and $0 \not \in \overline{A}$, we can choose $V$ and $W$ with $W+W+W \subseteq  \intt B \cap (G\setminus A)$.

Let us take $S_0:= A \cap B$ and $S_\ell:= B+\ell$ for all $0 \ne \ell \in L$.
According to Lemma~\ref{l:signswap}, to these $S_\ell$ and the chosen $V$, $W$ there exist positive definite functions $k_\ell$ with the properties that $k_\ell(0) = 2\lambda(S_\ell+W)/\lambda(V) \le 2 \lambda(B+W)/\lambda(V)$ (here we have equalities except for $\ell=0$), $(k_\ell)_{+} \ne 0$ only on $W$, $(k_\ell)_{-} \ne 0$ only on $S_\ell+W+W$, disjoint from $W$, and $k_\ell \equiv -1$ on $S_\ell$.

Next we take $p:=\sum_{\ell \in L} c_\ell k_{\ell}$.
The series here converges absolutely in the supremum-norm by the estimate  $|k_\ell|\leq k_l(0)\leq  2 \lambda(B+W)/\lambda(V)$ for each $\ell\in L$ and by the summability of the $c_\ell$. So $p$ is a continuous function. But we also need that $p\in X$.  To see this we argue as follows. Since $L$ is discrete and $B-B-(W+W)$ is relatively compact, there are only finitely many $\ell\in L$ that belong to $B-B-(W+W)$. Let $\ell_1,\dots,\ell_N\in L$ be these finitely many lattice points.

Let $\ell\in L\setminus\{0\}$ be fixed and take $x\in S_{\ell}=B+\ell$ arbitrarily. Let $m\in L$ be such that $k_{m}(x)\neq 0$, then we must have either $x\in W$ (if $k_{m}(x)>0$) or $x\in S_m+W+W$  (if $k_{m}(x)<0$), see the properties (i) and (ii) of the function $k_{m}$. Now, since $\ell\neq 0$ and so $x\in S_{\ell}\subseteq G\setminus W$,  in fact $x\in S_m+W+W\subseteq (B+m)+W+W$ must hold. This means that $x\in (B+l)\cap (B+m)+W+W$, i.e., $m-\ell\in B-B-(W+W)$, implying that $m=\ell+\ell_j$ for some $j\in \{1,\dots N\}$. Altogether we obtain that for $x\in S_\ell=B+\ell$ ($\ell\neq 0$) the following estimate is valid:
\[
|p(x)|\leq \sum_{n\in L} c_n|k_{n}(x)|\leq \sum_{j=1}^N c_{\ell+\ell_j}|k_{\ell+\ell_j}(x)|.
\]
Whence it follows that
\[
\|p|_{B+\ell}\|_\infty\leq  \frac{2 \lambda(B+W)}{\lambda(V)}\sum_{j=1}^N c_{\ell+\ell_j},
\]
and hence
\[
\sum_{\ell\in L}\|p|_{B+\ell}\|_\infty\leq \|p_{B}\|_\infty+\frac{2 \lambda(B+W)}{\lambda(V)}\sum_{\ell\in L}\sum_{j=1}^N c_{\ell+\ell_j}=  \frac{2N \lambda(B+W)}{\lambda(V)}\sum_{\ell\in L}c_\ell<\infty,
\]
i.e., indeed one has $p\in X$.

\medskip\noindent  As each $k_\ell \le 0$ outside of $W$, we obviously have on $G\setminus W$ that $p \le c_\ell k_\ell$, for whichever $\ell \in L$. In particular, at any $y \in S_\ell$, necessarily not belonging to $W$, we have $p(y)\le c_\ell k_\ell(y) = - c_\ell = -\|f|_{B+\ell}\|_\infty \le f(y)$. That is, we are led to
\[
p(y) \le f(y) \qquad (y \in \cup_{\ell \in L} S_\ell =S_0 \cup(G\setminus B) \supset A).
\]
The proof is complete.
\end{proof}

\begin{remark} The analogous, but simpler argument works when the original sign prescription encoded in the cone $Q$ is $f|_A\ge 0$. In that case we need $p\ge f$, which case was described in Ga\'al, R\'ev\'esz \cite{Zeitschrift}, too. It works similarly as above without using negative signs at all, that is, with $s_x:= 2r + T_x r + T_{-x} r $ replacing $r_x$, and deriving the respective upper, instead of lower, estimates on $A$.
\end{remark}

Now we can apply the following result of Jeyakumar and Wolkowicz, see  \cite[Lemma 2.2 (a)]{JeyaWolf} where a proof is given with an essential use of a result from Attouch, Brezis \cite{Brezis} for subdifferentials of convex functions.

\begin{lemma}[Jeyakumar-Wolkowicz]\label{l:Jeya_Wolk} If $R, S$ are closed convex sets, $0 \in R \cap S$, and the cone $\cone(R-S)$ generated by $R-S$ is a closed subspace, then $(R \cap S)^*=R^* + S^*$.
\end{lemma}
 For a closed subset $A\subseteq G$ let us write
\[
M_{+}(A;\RR)=\{\psi\in M(G;\RR):\psi \text{ is a positive Radon measure, $\supp(\psi)\subseteq A$}\},
\]
and recall that $\OO$ is the set of odd Radon measures in $M(G;\RR)$, while $\MM$ is the set of real-valued positive definite Radon measures in in $M(G;\RR)$.
\begin{proposition}\label{prop:coop}
	We have for an arbitrary set $A \subseteq  G$
	\[
	(P\cap Q_A)^*=P^*+Q_A^*  = - M_{+}(\oA;\RR) + \MM +\OO.
	\]
\end{proposition}
\begin{proof} According to Theorem \ref{t:Ppluslemma}, $P-Q_A= X$, the whole space. Therefore, Lemma \ref{l:Jeya_Wolk} applies and $(P \cap Q_A)^*= P^*+Q_A^*$. These dual cones, on the other hand, were described above in Propositions \ref{prop:Mdual} and \ref{prop:Pdual}, respectively. This concludes the proof.
\end{proof}

\section{The dual of the Delsarte problem and strong duality}
\label{sec:main}

 Let $G$ be a compactly generated locally compact  Abelian group. Consider the Banach spaces $X=C^{\infty,1}(G;\RR)$ and $\ X'=M(G;\RR)$.
 For subsets $\Omega, \Theta\subseteq G$, and a linear functional $\si\in X'$ we define
 \[
\FF_G^\si(\Omega):= \{ f \in C^{\infty,1}(G;\RR) ~:~ f \gg 0, f|_{\Omega^c} \le 0, \langle f,\si \rangle =1 \}
\]
and 
\[
\MM_G(\Theta):= \{ \mu \in M(G;\RR): \mu=\nu-\kappa, \, \kappa\geq 0, \,  \supp(\kappa)\subseteq  \overline{\Theta}, \, \text{$\nu$ real-valued, }\nu \gg 0\} .
\]
\begin{theorem}\label{thm:main}
	Suppose  $G$ is a compactly generated locally compact  Abelian group. Let $\Omega\subseteq G$ be a symmetric subset with $0\in \intt \Omega$, and let $\rh,\si\in X'=M(G;\RR)$, where  $\si$ is assumed to be  strictly positive definite and $\rh$ is assumed to be even.

Consider the linear programming ``primal'' extremal problem
\[
\alpha_\si^\rh (\Omega):= \inf \{\langle f,\rh\rangle~:~f \in \FF_G^\si(\Omega)\}.
\]
Then its linear programming dual problem is
\[
\omega_\si^\rh(\overline{\Omega^c}):= \sup \{ s \in \RR ~:~\rh- s \si  \in \MM_G(\Omega^c) \}.
\]
Moreover, there is no duality gap between the two problems.
\end{theorem}
\begin{proof}
We set $P:=\{f\in C^{\infty,1}(G;\RR):f\gg 0\}$ and $Q:=Q_{\Omega^c}:=\{f\in C^{\infty,1}(G;\RR): f|_{\Omega^c}\leq 0\}$. Then with $C=P\cap Q$ and $\mathcal{H}_\si:=\si^{-1}(\{1\})$ we have  $\FF_G^\si(\Omega)=C\cap \mathcal{H}_\si$.
By Proposition \ref{prop:coop} we obtain
\begin{equation}\label{eq:pcjo}
C^*=(P\cap Q)^*=P^*+Q^*=- M_{+}(\overline{\Omega^c};\RR) + \MM +\OO
\end{equation}
with $\OO$ the set of odd Radon measures in $M(G;\RR)$ and
	\[
	M_{+}(\overline{\Omega^c};\RR)=\{\psi\in M(G;\RR):\psi \text{ is a positive Radon measure, $\supp(\psi)\subseteq \overline{\Omega^c}$}\}.
	\]
Since $\si$ is assumed to be strictly positive definite, we obtain by Lemma \ref{lem:2ao} that
\[
\alpha_\si^\rh (\Omega)= \sup \{ s \in \RR ~:~\rh- s \si  \in C^* \},
\]
and then by \eqref{eq:pcjo}
\begin{equation}\label{eq:alphajo}
\alpha_\si^\rh (\Omega)=\sup \{ s \in \RR ~:~\rh- s\si \in - M_{+}(\overline{\Omega^c};\RR) + \MM +\OO\}.
\end{equation}

But since $\rh$ and $\si$ are even (the latter is so, because $\si$ is real-valued positive definite) we have for every $s\in \RR$ that 
\[
\rh- s \si  \in \MM_G(\Omega^c) \quad\Longleftrightarrow\quad \rh- s \si \in P^*+Q^*=- M_{+}(\overline{\Omega^c};\RR) + \MM +\OO.
\]
Indeed, the implication ``$\Rightarrow$'' in this equivalence is trivial. To see the other one suppose that $\rh- s \si \in- M_{+}(\overline{\Omega^c};\RR) + \MM+\OO$, so $\rh- s \si=\nu+\kappa+\mu$, with $\nu\gg 0$ and real, $\kappa\in -M_+(\overline{\Omega^c};\RR)$ and $\mu\in \OO$. We can write $\kappa=\kappa_e+\kappa_o$ with $\kappa_e$ even and $\kappa_o$ odd. But then $\kappa_e\in-M_+(\overline{\Omega^c};\RR)$ because, by assumption, $\overline{\Omega^c}$ is symmetric. Therefore, as $\rh- s \si $, $\kappa_e$, $\nu$ are even, $\kappa_o, \mu$ are odd, and $\rh- s \si=\nu+\kappa_e+\kappa_o+\mu$, we must have $\kappa_o+\mu=0$. The asserted equivalence is proven. Whence we obtain
\begin{align*}
\sup \{ s \in \RR ~:~\rh- s\si \in \MM_G(\Omega^c) \}&
=\sup \{ s \in \RR ~:~\rh- s\si \in - M_{+}(\overline{\Omega^c};\RR) + \MM +\OO\}\\
&=\alpha_\si^\rh (\Omega),
\end{align*}
the last equality being \eqref{eq:alphajo}.
\end{proof}

\begin{remark}\label{rem:mainmain}
 If $G$ is not compactly generated one should adjust the previous setting and argumentation as follows.  Consider a compactly generated, open and closed subgroup $G_0$ in $G$. Then $G/G_0$ is discrete, and we take $\Lambda\subseteq G$ a complete set of representatives of the coset space. In $G_0$ we find a lattice $L$ and the corresponding tile $B$, as  given in Proposition~\ref{prop:compact tile}. The definition of the space $X$ needs to be modified as follows: For $f\in C(G;\RR)$ we set
 \[
\|f\|_X:=\sum_{\ell\in L, m\in \Lambda}\|f|_{m+\ell+B}\|_\infty,
 \]
 and
\[
X:=C^{\infty,1}(G;\RR):=\{f\in C(G;\RR):\|f\|_X<\infty\},
\]
which becomes a Banach space with the norm $\|\cdot \|_X$. We note that every $f\in X$ vanishes outside a $\sigma$-compact set. The dual space of $X$ is the Banach space $M(G,\RR)$ of translation bounded Radon measures on $G$.
All the statements in Sections \ref{sec:lcaback} and \ref{sec:Dual-Cone-Intersection} remain true in this setting. In particular, the dual cones of $P$ and $Q_A$ can be determined and the dual cone formula is valid.
\end{remark}

By this remark we can formulate the following version of the main result of this paper, using the notation introduced above. The proof requires no modification as compared to the one of Theorem \ref{thm:main}.

\begin{theorem}\label{thm:mainmain}
	Suppose  $G$ is a  locally compact  Abelian group. Let $\Omega\subseteq G$ be a symmetric subset with $0\in \intt \Omega$, and let $\rh,\si\in X'=M(G;\RR)$, where  $\si$ is assumed to be  strictly positive definite and $\rh$ is assumed to be even.

Consider the linear programming ``primal'' extremal problem
\[
\alpha_\si^\rh (\Omega):= \inf \{\langle f,\rh\rangle~:~f \in \FF_G^\si(\Omega)\}.
\]
Then its linear programming dual problem is
\[
\omega_\si^\rh(\overline{\Omega^c}):= \sup \{ s \in \RR ~:~\rh- s \si  \in \MM_G(\Omega^c) \}.
\]
Moreover, there is no duality gap between the two problems.
\end{theorem}

\begin{remark}[Delsarte constant]
In Theorem \ref{thm:mainmain} let us take specifically $\rh$ as the negative of the Haar integral, $f\mapsto -\int f \dd\haar$ (thus $\rh=-\haar$, after identifying  functionals with measures) and $\si:=\delta_0$ the point evaluation at the neutral element $0$, $f\mapsto f(0)$. Both functionals fulfill the requirements of Theorem \ref{thm:mainmain}. For the value of the primal problem we obtain
\begin{align*}
\alpha^\rh_\si(\Omega) &=\inf\Bigl\{-\int_G f\dd\haar:f\gg 0,\, f|_{\Omega^c}\leq 0,\, f(0)=1\Bigr\}\\
&=-\sup\Bigl\{\int_G f\dd\haar:f\gg 0,\,f|_{\Omega^c}\leq 0 ,\,f(0)=1\Bigr\},
\end{align*}
so $-\alpha^\rh_\si(\Omega)= D_G(X,\Omega)$ is the Delsarte constant, as given in  Definition \ref{def:Delsarte}.  Theorem \ref{thm:mainmain} then yields
\begin{align*}
D_G(X,\Omega)&=- \sup \{ s \in \RR ~:~-\haar- s \delta_0  \in \MM_G(\Omega^c) \}\\
&=\inf \{- s \in \RR ~:~-\haar- s \delta_0  \in \MM_G(\Omega^c) \}\\
&=\inf \{s \in \RR ~:~-\haar+s \delta_0  \in \MM_G(\Omega^c) \}.
\end{align*}
\end{remark}
\begin{remark}[ Discrete groups]
If $G$ is discrete, then $X=C^{\infty,1}(G;\RR)=\ell^1(G;\RR)$ and $X'=\ell^\infty(G;\RR)$. Theorem \ref{thm:mainmain}  also yields  Theorem \ref{thm:revvir} for the case when $\Omega_-=G$ and $\Omega_+=\Omega$.
\end{remark}

Next suppose that $G$ is compact.
Then $X=C^{\infty,1}(G;\RR)=C(G;\RR)$ and $X'=M(G;\RR)$ the space of regular finite Borel measures on $G$. We thus obtain:
\begin{theorem}[$G$ compact] Let $G$ be a compact Abelian group with neutral element $0$, and let $\Omega \subseteq G$ be a set with $0\in \intt\Omega$. Consider the Banach spaces $X=C(G;\RR)$ and  $\rh,\si\in X'=M(G;\RR)$, where  $\si$ is assumed to be  strictly positive definite  and $\rh$ is assumed to be even.
Consider the linear programming ``primal'' extremal problem
\[
\alpha_\si^\rh (\Omega):= \inf \{\langle f,\rh\rangle~:~f \in \FF_G^\si(\Omega)\}.
\]
Then its linear programming dual problem is
\[
\omega_\si^\rh(\Omega^c):= \sup \{ s \in \RR ~:~\rh- s \si  \in \MM_G(\Omega^c) \}.
\]
Moreover, there is no duality gap between the two problems.
\end{theorem}
This is partly a generalization of the result in \cite{UFO} in the extent that here we allow more general weights $\rh$ and more general normalizing functionals $\si$. On the other hand, \cite{UFO} is from a certain viewpoint far more general as it handles two-sided sign restrictions on the primal problem, similarly to the one in the discrete case in Section \ref{sec:discrete} here. Moreover, the results there cover the case of compact Gelfand pairs.

\section*{Acknowledgements}%

This research was partially supported by the DAAD-Tempus PPP Grant 57448965 ``Harmonic Analysis and Extremal Problems''.

Elena E. Berdysheva was supported  in part by the University of Cape Town's Research Committee (URC).

Elena E. Berdysheva and Mita D. Ramabulana thank the HUN-REN R\'enyi Institute of Mathematics for hospitality during their respective visits.


Marcell Ga\'al was supported by the National Research, Development and Innovation Office -- NKFIH Reg. No.'s K-115383 and K-128972, and also by the Ministry for Innovation and Technology, Hungary throughout Grant TUDFO/47138-1/2019-ITM.

Mita D. Ramabulana was supported by the Carnegie DEAL 3 Postdoctoral Fellowship.

Szil\'ard Gy. Révész was supported in part by the Hungarian National Research, Development and Innovation Fund projects \# K-119528, K-132097, K-146387, K-147153 and Excellence No.
151341.
%

\end{document}